\documentclass[reqno,11pt]{amsart}

\usepackage{amsmath}
\usepackage{amsthm}
\usepackage{amsfonts}
\usepackage{amssymb,bbm}
\usepackage{enumerate}
\usepackage{graphicx}
\usepackage{color}
\usepackage[usenames,dvipsnames]{xcolor}
\usepackage{verbatim}
\usepackage{hyperref}

\newcommand\abs[1]{\left|#1\right|}
\newcommand\floor[1]{\left\lfloor#1\right\rfloor}

\numberwithin{equation}{section}

\newcommand{\Prob}{\mathbb{P}}

\renewcommand\Re{\operatorname{Re}}
\renewcommand\Im{\operatorname{Im}}
\newcommand{\eps}{\varepsilon}

\newcommand{\Le}{\mathcal{L}}

\newcommand{\mat}{\mathbf}

\newcommand{\1}{\mathbf{1}}
\renewcommand{\Pr}{\mathbb P}

\newcommand{\pfrac}[2]{\left(\frac{#1}{#2}\right)}

\theoremstyle{plain}
  \newtheorem{theorem}{Theorem}[section]
  \newtheorem{conjecture}[theorem]{Conjecture}
  \newtheorem{lemma}[theorem]{Lemma}

  \newtheorem{proposition}[theorem]{Proposition}

\theoremstyle{definition}
  \newtheorem{definition}[theorem]{Definition}
  \newtheorem{example}[theorem]{Example}

\theoremstyle{remark}
  \newtheorem{remark}[theorem]{Remark}

\begin{document}
\title{Low-degree factors of random polynomials}

\author[S. O'Rourke]{Sean O'Rourke}
\address{Department of Mathematics, University of Colorado at Boulder, Boulder, CO 80309  }
\email{sean.d.orourke@colorado.edu}

\author[P. M. Wood]{Philip Matchett Wood}
\thanks{The second author was partially supported by National Security Agency (NSA) Young Investigator Grant numbers H98230-14-1-0149 and H98230-16-1-0301.} 
\address{Department of Mathematics, University of Wisconsin-Madison, 480 Lincoln Dr., Madison, WI 53706 }
\email{pmwood@math.wisc.edu}

\begin{abstract}
We study the probability that a monic polynomial with integer coefficients has a low-degree factor over the integers, which is equivalent to having a low-degree algebraic root.  It is known in certain cases that random polynomials with integer coefficients are very likely to be irreducible, and our project can be viewed as part of a general program of testing whether this is a universal behavior exhibited by many random polynomial models.  
Our main result shows that pointwise delocalization of the roots of a random polynomial can be used to imply that the polynomial is unlikely to have a low-degree factor over the integers.  We apply our main result to a number of models of random polynomials, including characteristic polynomials of random matrices, where strong delocalization results are known.
%


\end{abstract}

\maketitle

\section{Introduction}\label{S:intro}

Consider the following question: is it true that a random monic polynomial with integer coefficients is irreducible with high probability?    For example, a version of Hilbert's Irreducibility Theorem\footnote{
See \cite{Z} for a  modern formulation of Hilbert's Irreducibility Theorem.}  
states that if $h_{n,N}$ is a monic polynomial in one variable of fixed degree $n$ where all coefficients except the degree $n$ coefficient are chosen independently and uniformly at random from among all integers in the interval $[-N, N]$, then the probability that $h_{n,N}$ is irreducible approaches $1$ in the limit as $N \to \infty$.  
This was first proved by van der Waerden in 1934 \cite{vdW1934}; and in fact, the probability that $h_{n,N}$ is reducible is of order $1/N$, which was proven by van der Waerden two years later \cite{vdW1936}. (The existence and value of the limiting constant was determined by Chela \cite{Ch} in 1963 in terms of the Riemann Zeta function.)  
%
Van der Waerden \cite{vdW1934,vdW1936} also showed that, with probability tending to $1$, the Galois group of the random polynomial $h_{n,N}$ is the full symmetric group $\mathfrak S_n$ on $n$ elements (which implies irreducibility) as $N \to \infty$.  
Estimates for the exact order for the probability that the Galois group is not $\mathfrak S_n$ have been improved since van der Waerden, first in 1955 and 1956 by Knobloch 
\cite{Kzum,K56}, then in 1973 by Gallagher \cite{XG} who applied the large sieve, followed by more recent progress in 2010 by Zwina \cite{Z}, in 2013 by Dietmann \cite{D}, and in 2015 by Rivin \cite{Rivin}
(see also
\cite{
Cohen1, Cohen2,
W}
and references therein).

How the random polynomial is generated matters, and there is a general heuristic that if the random integer coefficients are generated so that ``elementary'' factorizations are avoided---for example, one ensures that the constant coefficient is not likely to be zero, in which case $x$ would be a factor of the polynomial $f(x)$---then the polynomial is very likely to be irreducible.  One can think of this heuristic as suggesting a kind of universality (see, for example,  \cite[Heuristic 1.1]{BBBSWW}), and in some specific instances, it has been conjectured that the behavior in Hilbert's Irreducibility Theorem extends to different settings, including when the degree $n$ is growing.  For example, one can define a random polynomial $g_{n}$ where the constant coefficient and the degree $n$ coefficient are equal to $1$, and all other coefficients are $0$ or $1$ independently with probability $1/2$.  In the limit as the degree $n$ goes to infinity (in contrast to the degree being fixed in Hilbert's Irreducibility Theorem and the results discussed above) it has been conjectured that, once again, the probability that $g_n$ is irreducible approaches one as $n\to\infty$ (see \cite{K01,OP})

The question of proving irreducibility in the case where the degree of the random polynomial tends to infinity and the support of the coefficients remains bounded (or bounded by a function of the degree) seems to be quite challenging.   For example, in the specific case of the polynomials $g_n$ described above, the current best result (due to Konyagin \cite{K01}) shows that the probability is bounded below by $c/\log n$, where $c$ is a positive constant, and as far as the authors know, there is not a result showing that the probability that $g_n$ is irreducible remains bounded away from zero as the degree increases, even though this probability is conjectured to approach 1. (Interestingly, Bary-Soroker and Kozma \cite{BSK} have proven that bivariate polynomials with independent $\pm 1$ do become irreducible with high probability as the degree increases, though the approach does not extend to a single-variable polynomial like $g_n$.) One key step in Konyagin's result \cite{K01} is showing that $g_n$ is unlikely to have a factor over the integers with degree up to $c n/\log n$, which is step towards proving irreducibility; note that showing that there is no factor over the integers of degree up to $n/2$ would prove irreducibility for a degree $n$ polynomial.  

In the current note, we show that the phenomenon of random polynomials having no factors over the integers with small degree is quite general, and in fact can be implied by pointwise delocalization of the roots of the random polynomial.  Generally speaking, we show that, for a random monic polynomial $f$ with integer coefficients, if $\sup_{z \in \mathbb{C}} \Pr(f(z)=0)$ is sufficiently small, 
%
%
then the probability of a low-degree factor over the integers is also small.  We refer to the quantity $\sup_{z \in \mathbb{C}} \Pr(f(z)=0)$ being small as \emph{pointwise delocalization}.  In particular, pointwise delocalization rules out the possibility that $f$ has a deterministic (or near deterministic) root.  More generally, pointwise delocalization can be viewed as measuring the probability that $f$ has some ``elementary'' factorization.  For instance, $\Prob(f(0) = 0)$ is the probability that $z$ is a factor of $f(z)$.

Our main result  provides useful bounds for random polynomials with correlated and highly dependent coefficients, because pointwise delocalization is a statement about the roots, rather than the coefficients.  
This is particularly useful, for example, when studying the characteristic polynomial of a random matrix: the coefficients are typically dependent and correlated, but often more is known about the roots, which are the eigenvalues of the matrix.

When $f$ is the characteristic polynomial of a square random matrix, we can often show that the pointwise delocalization condition holds by using sufficiently general results which bound the probability that the matrix is singular or has a very small singular value. 
In Section \ref{sec:examples}, we consider various models of random polynomials and random matrices for which good pointwise delocalization results are known.  For example, we show that for any $\epsilon > 0$ and for an $n$ by $n$ random matrix with each entry $+1$ or $-1$ independently with probability $1/2$, the characteristic polynomial factors over the integers with a factor of degree at most $n^{1/2 -\epsilon}$ with probability at most $\left(\frac{1}{\sqrt 2} + o(1)\right)^n$ (see Theorem~\ref{thm:iidpm1}).
    
We begin by fixing some terminology and notation.  If $F$ is a field, a polynomial with coefficients in $F$ is \emph{irreducible over $F$} if the polynomial is nonconstant and cannot be factored into the product of two nonconstant polynomials with coefficients in $F$.  More generally, a polynomial with coefficients in a unique factorization domain $R$ (for example, the integers) is said to be \emph{irreducible over $R$} if it is an irreducible element of the polynomial ring $R[x]$, meaning that the polynomial is nonzero, is not invertible, and cannot be written as the product of two non-invertible polynomials with coefficients in $R$.  Irreducibility of a polynomial over a ring $R$ generalizes the definition given for the case of coefficients in a field because, in the field case, the nonconstant polynomials are exactly the polynomials that are non-invertible and nonzero.  We say $f$ is \emph{reducible over $R$} if $f$ is not irreducible over $R$.  

Recall that an \emph{algebraic number} is a possibly complex number that is a root of a finite, nonzero polynomial in one variable with rational coefficients (or equivalently, by clearing the denominators, with integer coefficients).  Given an algebraic number $\alpha$, there is a unique monic polynomial with rational coefficients of least degree that has the number as a root. This polynomial is called the \emph{minimal polynomial} for $\alpha$, and if $\alpha$ is a root of a polynomial $f$ with rational coefficients, then the minimal polynomial for $\alpha$ divides $f$ over the rationals. 
If the minimal polynomial has degree $k$, then the algebraic number $\alpha$ is said to be of degree $k$.  For instance, an algebraic number of degree one is a rational number.  An \emph{algebraic integer} is an algebraic number that is a root of a polynomial with integer coefficients with leading coefficient $1$ (a monic polynomial).   The question of whether a monic polynomial $f$ with integer coefficients has an irreducible degree $k$ factor when factored over the rationals is thus equivalent to whether $f$ has a root $\alpha$ that is an algebraic number of degree $k$; in fact, by Gauss's Lemma (see for instance \cite{DF}), $f$ being monic implies that $\alpha$ is an algebraic integer.

Let $f$ be a polynomial of degree $n$ over $\mathbb{C}$.  We let $\lambda_1(f), \ldots, \lambda_n(f) \in \mathbb{C}$ denote the zeros (counted with multiplicity) of $f$, and we define 
\begin{equation} \label{def:Lambda}
	\Lambda(f) := \{ \lambda_1(f), \ldots, \lambda_n(f) \} 
\end{equation}
to be the set of zeros of $f$.  

\subsection{Models of random monic polynomials with integer coefficients}

As mentioned above, there are many ensembles of random polynomials.  We begin with the most general ensemble of random monic polynomials with integer coefficients.  

\begin{definition}[Random monic polynomial] \label{def:monicpoly}
We say $f(z) := z^n + \xi_{n-1} z^{n-1} + \cdots + \xi_1 z + \xi_0$ is a degree $n$ \emph{random monic polynomial with integer coefficients} if $\xi_{n-1}, \ldots, \xi_0$ are integer-valued random variables (not necessarily independent).  
\end{definition}

We emphasize that the integer-valued random variables $\xi_{n-1}, \ldots, \xi_0$ are not assumed to be independent or identically distributed.  There are many examples of such random polynomials.  

\begin{example}[Independent Rademacher coefficients] \label{example:iidcoef}
Let $\xi_0, \ldots, \xi_{n-1}$ be independent Rademacher random variables, which take the values $+1$ or $-1$ with equal probability.  Then $f(z) := z^n + \xi_{n-1} z^{n-1} + \cdots + \xi_1 z + \xi_0$ is a random monic polynomial with integer coefficients.  More generally, one can consider the case when $\xi_0, \ldots, \xi_{n-1}$ are independent and identically distributed (iid) copies of an integer-valued random variable (not necessarily Rademacher); see Example \ref{example:uniform} below for one such example.    
\end{example}

\begin{example}[Independent uniform coefficients] \label{example:uniform}
Let $N \in \mathbb{N}$ be a given parameter.  Let $\xi_0, \ldots, \xi_{n-1}$ be independent and identically distributed (iid) random variables uniformly distributed on the discrete set $\{0, 1, \ldots, N\}$.  Then $f(z) := z^n + \xi_{n-1}z^{n-1} + \cdots + \xi_1 z + \xi_0$ is a random monic polynomial with integer coefficients.  
\end{example}

\begin{example}[Characteristic polynomial of random matrices]
Let $\xi$ be an integer-valued random variable, and let $\mat X$ be an $n \times n$ random matrix whose entries are iid copies of $\xi$.  Then the characteristic polynomial $f(z) := \det (z \mat I - \mat X)$ is a random monic polynomial with integer coefficients.  Here, $\mat I$ denotes the identity matrix.  
\end{example}

\begin{example}[Random permutation matrices]
Let $\pi$ be a random permutation on $\{1, \ldots, n\}$ uniformly sampled from all $n!$ permutations.  Let $\mat P_\pi$ denote the corresponding permutation matrix, i.e., the $(i,j)$-entry of $\mat P_{\pi}$ is one if $i = \pi(j)$ and zero otherwise.  Clearly, $\mat P_{\pi}$ is an orthogonal matrix.  The permutation $\pi$ may be written as a product of $\ell$ disjoint cycles with lengths $c_1,\dots,c_\ell$.  Let $f_{\pi}$ denote the characteristic polynomial of $\mat P_{\pi}$.  Then, as can be seen by reordering the rows and columns of $\mat P_\pi$ so that it is block diagonal, we have 
$$ f_{\pi}(z) := \det(z \mat I - \mat P_{\pi})= \prod_{j=1}^\ell (z^{c_j}-1),$$
where $\mat I$ is the identity matrix.  Clearly $1$ is always a root of $f_{\pi}$, making $z-1$ a factor and $f_{\pi}$ reducible.  In addition, $f_{\pi}$ will have other (possibly repeated) factors as well if $n$ is composite or if the number of cycles $\ell$ is at least 2.  One way to measure randomness in the roots of a random polynomial is testing whether the polynomial has any double roots.  For example, Tao and Vu \cite{TV} have shown that the spectrum of a random real symmetric $n$ by $n$ matrix with independent entries contains no double roots with probability tending to $1$ as $n$ increases (see also 
\cite{FS, PSZ}
for a related question on another class of random polynomials).  For contrast, in the case of the characteristic polynomial of a random permutation matrix, the probability that the spectrum contains no double roots is the same as the probability of the permutation having only one cycle, which occurs with probability $1/n$ and tends to zero, rather than 1.

\end{example}

\begin{example}[Erd\H os--R\'enyi random graphs]
Let $G(n,p)$ be the Erd\"os--R\'enyi random graph on $n$ vertices with edge density $p$.  That is, $G(n,p)$ is a simple graph on $n$ vertices (which we shall label as $\{1, \ldots, n\}$) such that each edge $\{i,j\}$ is in $G(n,p)$ with probability $p$, independent of other edges.  In the special case when $p=1/2$, one can view $G(n,1/2)$ as a random graph selected uniformly among all $2^{\binom{n}{2}}$ simple graphs on $n$ vertices.  The random graph $G(n,p)$ can be defined by its adjacency matrix $\mat A_n$, which is a real symmetric matrix with entry $(i,j)$ equal to 1 if there is an edge between vertices $i$ and $j$, and the entry equal to zero otherwise.  It is widely believed (and numerical evidence suggests) that the characteristic polynomial of $\mat A_n$ is irreducible with probability tending to one as $n \to \infty$.  We discuss this example more in Section \ref{ss:simplegraphs} and Section \ref{sec:control}.  
\end{example}

We have chosen to focus on monic polynomials, but the question of irreducibility can also be asked for non-monic random polynomials with integer coefficients (or equivalently, by dividing by the leading coefficient, for random monic polynomials with rational coefficients).  For fixed-degree polynomials with independent coefficients, this question was addressed by Kuba \cite{Kuba}.  When the degree tends to infinity, we again expect the answer to depend on the random polynomial model.

\subsection{Main results}
In this paper, we focus on the algebraic degree of the roots of a random monic polynomial $f$.  
Our main result below bounds above the probability that $f$ has an algebraic root of degree $k$, for some given value of $1 \leq k \leq n$, which is related to the question of irreducibility since a monic polynomial with integer coefficients is irreducible if and only if its roots are all algebraic of degree $n$.  We expect many random monic polynomial models to yield irreducible polynomials with high probability, and so intuitively, algebraic roots of small degree should be rare.  

\begin{theorem} \label{thm:main}
Let $f$ be a degree $n$ random monic polynomial with integer coefficients (as in Definition \ref{def:monicpoly}).  Let $M > 0$ and $2 \leq k \leq n$. Take $\Omega \subseteq \{z \in \mathbb{C} : |z| \leq M \}$, and
suppose there exists $p \in [0,1]$ such that
\begin{equation} \label{eq:zerobnd}
	\sup_{z \in \Omega} \Prob ( f(z) = 0 ) \leq p 
\end{equation}
(in other words, pointwise delocalization holds on $\Omega$).
Then, 
the probability that $f$ has an algebraic root of degree at most $k$ in $\Omega$ is at most
\begin{equation} \label{eq:sumprobbnd}
	p (eM)^{k^2}+ \Prob(\abs{\lambda_i(f)} > M\mbox{ for some } i),
\end{equation}
where $\lambda_1(f), \ldots, \lambda_n(f)$ are the roots of $f$.  If $k=1$, the result holds if $p(eM)^{k^2}$ is replaced with $p(3M)$ in \eqref{eq:sumprobbnd}.
\end{theorem}

For Theorem~\ref{thm:main} to be useful, one needs to show that the bound \eqref{eq:sumprobbnd} is small. In Lemma~\ref{lem:collected-bounds} we collect bounds on $p(eM)^{k^2}$ that hold for specific random polynomial models that we will discuss in Section~\ref{sec:examples}.

\begin{lemma}\label{lem:collected-bounds}
We have the following bounds on $p(eM)^{k^2}$ for various values of $p$ (the pointwise delocalization parameter), $M$ (the radius containing $\Omega$), and $k$ (the degree).
\begin{enumerate}[(i)]
\item If $p=O(1/\sqrt n)$ and $M=2$ and $k \le \sqrt{\frac{\log n}{4}}$, then $p(eM)^{k^2} = o(1)$.
\item If $p=\left(\frac1{\sqrt 2} + o(1)\right)^n$ and $M=n$ and $k=n^{1/2 -\epsilon}$ for some $\epsilon >0$, then $p(eM)^{k^2} =\left(\frac1{\sqrt 2} + o(1)\right)^n$ (the two $o(1)$ terms differ).
\item If $p = 2e^{-n^c}$ for some $0<c<1$, $M = C\sqrt n$ for some $C >0$, and $k \le n^{c'}$ for $c'< c/2$, then $p(eM)^{k^2} \le 2 \exp\left({-\pfrac{2}{3}n^c}\right)$ for sufficiently large $n$.
\item Let $B>0$ and $m\ge 1$, and take $M = n^m$ and $k \geq 1$ constant.  Then there exists $B' > 0$ (depending only on $B, m$, and $k$) such that if $p=n^{-B'}$, then $p(eM)^{k^2} \le n^{-B}$ for sufficiently large $n$.
\end{enumerate}

\end{lemma}

\subsection{Random polynomials over finite fields}
There are, of course, many other ensembles of random polynomials one can consider.  For instance, one can study monic polynomials over the finite field $\mathbb{F}_q$, where $q$ is a power of a prime.  Indeed, there are $q^n$ monic polynomials of degree $n$ over $\mathbb{F}_q$, and we can consider selecting one uniformly at random.  Using Galois theory for finite fields and M\"{o}bius inversion (see \cite[Section 14.3]{DF}), one can show that the number of degree $n$ irreducible polynomials over $\mathbb F_q$ is
$$\frac1n \sum_{d | n}\mu(d) q^{n/d},$$
where $\mu$ is the M\"{o}bius function.  Thus, the probability that a
randomly selected degree $n$ monic polynomial over $\mathbb{F}_q$ is irreducible is
$$ \frac{1}{nq^n}  \sum_{d | n} \mu(d) q^{n/d} = \frac{1}{n} + O(q^{-n/2}), $$
(using the coarse bound  $\abs{\mu(d)} \le 1$) for any $n$ and $q$.  Thus, in a finite field, a degree $n$ polynomial chosen uniformly at random is irreducible only with probability close to $1/n$.  This contrasts sharply with the case of polynomials over the integers, where Hilbert's Irreducibility Theorem shows that at randomly chosen polynomial is very likely to be irreducible (see, for example, \cite{Z}).


\subsection{Overview and outline}
The paper is organized as follows.  In Section~\ref{sec:examples}, we give some example applications of our main results, including the cases of random polynomials with iid coefficients, the characteristic polynomial of random matrices (non-symmetric, non-symmetric sparse, symmetric, and elliptical), and adjacency matrices of random graphs (directed, undirected, and fixed outdegree).   Often we will consider the case where the underlying random variables are Rademacher $\pm 1$ for simplicity.  Section~\ref{sec:control} motivates the model of random polynomials studied in this paper by illustrating a connection that exists between irreducible random polynomials, random graphs, and control theory on large scale graphs and networks.  Theorem~\ref{thm:main} and Lemma~\ref{lem:collected-bounds} are proven in Section~\ref{sec:main}.  
Finally, Section~\ref{sec:iidpoly} contains the proof for one of the applications discussed in Section~\ref{sec:examples}.

\subsection{Notation}
We use asymptotic notation (such as $O, o$) under the assumption that $n \to \infty$.  In particular, $o(1)$ denotes a term which tends to zero as $n \to \infty$.  
Let $[n] := \{1, \ldots, n\}$ denote the discrete interval.  We let $\sqrt{-1}$ denote the imaginary unit and reserve $i$ as an index.  For a finite set $S$, we use $|S|$ to denote the cardinality of $S$.  
For a vector $v$, we use $\|v\|$ for the Euclidean norm.  We let $u^\mathrm{T} v = u \cdot v$ denote the dot product between two vectors $u, v \in \mathbb{R}^n$.  For a matrix $\mat A$, we let $\| \mat A\|$ denote the spectral norm, i.e., $\|\mat A\|$ is the largest singular value of $\mat A$.  We let $\mat I_n$ denote the $n \times n$ identity matrix; often we will drop the subscript $n$ when its size can be deduced from context.  For a polynomial $f$, $\deg(f)$ denotes the degree of $f$.

\begin{figure}
\includegraphics[scale=.35]{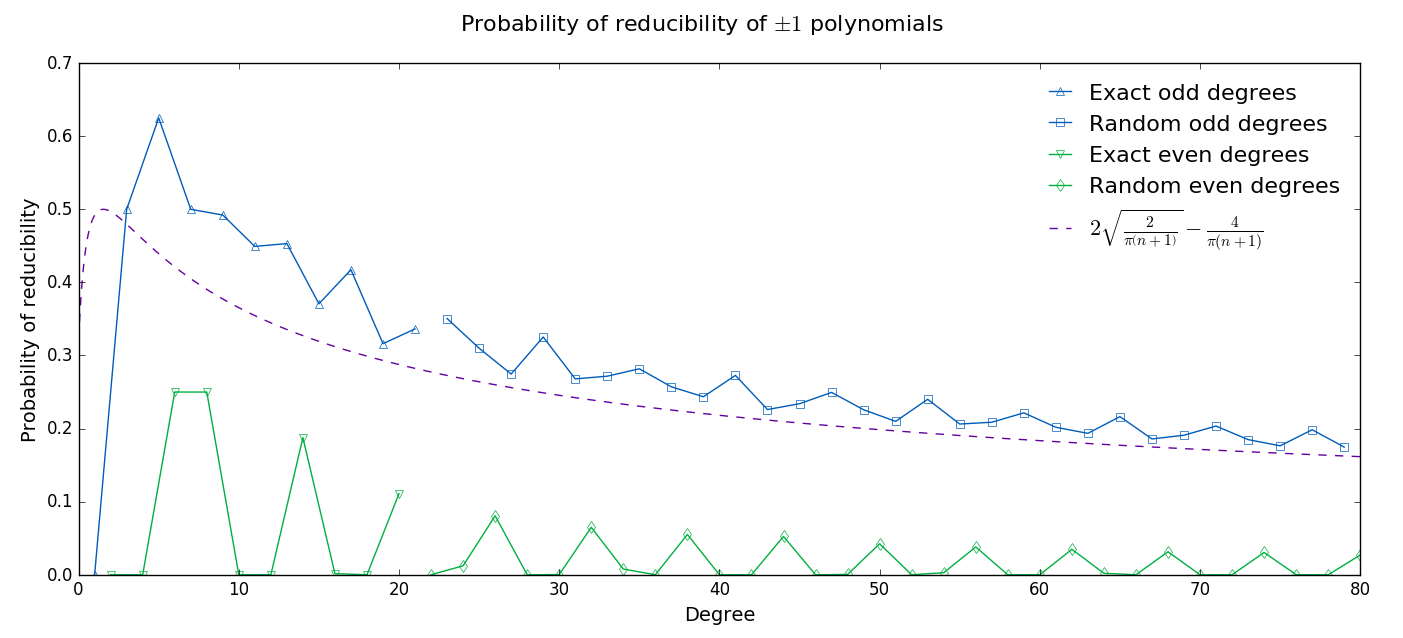}
\caption{Quoted from \cite[Figure~3]{BBBSWW}, the above shows the probability that 
 $f_n(x) = x^n + \xi_{n-1}x^{n-1} + ... + \xi_1 x + \xi_0$ is irreducible, where the $\xi_i$ are $+1$ or $-1$ independently with probability $1/2$.  
For degree up to $n=21$, the probability was computed exactly by exhaustively generating all $2^{n}$ polynomials of the specified degree and checking each one for reducibility using the \texttt{IsIrreducible()} function in Magma. For degree 22 up to 80, the probability was estimated (again using Magma) by generating 150,000,000 random polynomials for each degree.
%
The curve 
$2\sqrt{\frac{2}{\pi( n+1)}}-\frac{4}{\pi (n+1)}$
is an asymptotic lower bound for the probability of reducibility when the degree is odd.
%
Figure produced by 
Christian Borst, 
Evan Boyd, 
Claire Brekken, 
and 
Samantha Solberg 
 (see \cite{BBBSWW}).}
\label{fig:pm1}
\end{figure}

\section{Example applications of the main results} \label{sec:examples}

We now specialize Theorem~\ref{thm:main} to some specific examples.  

\subsection{Random polynomials with iid coefficients}
We now consider Example \ref{example:iidcoef}, where the coefficients of $f$ are iid random variables.  

\begin{theorem}[Random polynomials with iid coefficients] \label{thm:iidpoly}
For each $n \geq 1$, let $f_n(z) = z^n + \xi_{n-1} z^{n-1} + \cdots + \xi_1 z + \xi_0$, where $\xi_0, \xi_1, \ldots$ are iid Rademacher random variables, which take the values $+ 1$ or $-1$ with equal probability.  Then the probability that $f_n(x)$ has an algebraic root of degree at most $\frac{n^{1/3}}{\log^3 n}$ is at most $O\pfrac{1}{\sqrt n}$.
\end{theorem}

\begin{remark}\label{rem:weaker_iidpoly}
Note that a weaker version of Theorem~\ref{thm:iidpoly} follows easily from Theorem~\ref{thm:main}. In particular, by Lemma~\ref{lemma:iidpoly122} all roots for $f_n$ have absolute value between $1/2$ and $2$, and by the Littlewood-Offord Theorem (see, e.g., \cite[Corollary~7.8]{TVaddcomb}) $\Pr(f_n(z)=0) \le O(1/\sqrt n)$ for any $z$.  Thus, Lemma~\ref{lem:collected-bounds}(i) combined with Theorem~\ref{thm:main} implies that $f_n$ has no algebraic roots of degree at most $\sqrt{\frac{\log n}{4}}$ with probability tending to $1$ as $n$ tends to infinity.
\end{remark}

We present a proof of Theorem~\ref{thm:iidpoly} in Section~\ref{sec:iidpoly}, and below we will comment on potential generalizations of Theroem~\ref{thm:iidpoly} and its connections to the work of Konyagin \cite{K01}. See Figure~\ref{fig:pm1} for numerical evidence suggesting that, in fact, the probability that $f_n$ is reducible goes to zero as $n \to \infty$.

Beyond Theorem~\ref{thm:iidpoly}, our methods can also be used when $\xi_0, \xi_1, \ldots$ are more general iid integer-valued random variables satisfying some technical assumptions.  However, a number of complications can arise in this case (e.g., zero is always a root of $f_n$ with probability $\Prob(\xi_0 = 0)$), and so
we focus on the Rademacher $\pm 1$ case for simplicity.  

In \cite{K01}, Konyagin studies the random degree $n$ polynomial $g_n$ which has $1$ for the constant coefficient and the degree $n$ coefficient, and every other coefficient is $0$ or $1$ independently with equal probability.  In particular, he shows that there are constants $c,C >0$ such that $g_n$ has a root that is an algebraic number with degree at most $ cn/\log n$ with probability at most $C/\sqrt n$.
%
Konyagin's approach in \cite{K01} can also be adapted to more general distributions of the random integer coefficients (see forthcoming work of Terlov \cite{Terlov}); however, the method seems to require independence of the coefficients, whereas an application of Theorem~\ref{thm:main} would allow for dependence, though at the expense of a weaker bound on the size of the low-degree factors.

Finally, one should note that elementary Galois theory can be used to prove that 
if $n+1$ is prime and $2$ generates the multiplicative group\newline  $\Big(\mathbb Z/(n+1)\Big)^\times$, then
\emph{every} random polynomial of degree $n$ with coefficients iid Rademacher $\pm 1$ random variables (as in Theorem~\ref{thm:iidpoly}) is in fact irreducible.\footnote{We thank Melanie Matchett Wood for describing the formulation and proof of this result.}
  One can prove this by considering the polynomials modulo 2, in which case $+1=-1 \mod 2$ and every polynomial is equal to $x^n+x^{n-1}+x^{n-2}+\dots+1$ (i.e., there is no randomness); thus every root of the polynomial modulo $2$ must be a $(n+1)$-st root of unity.  To complete the argument, one can use the fact that $\mathbb F_{2^n}$ has cyclic multiplicative group and the fact that the Galois group $\operatorname{Gal}(\mathbb F_{2^n}/ \mathbb F_2)$ is also cyclic and generated by the Frobenius endomorphism $x\mapsto x^2$ (see \cite{DF}).  Interestingly, letting $p=n+1$ be a prime, Artin's Conjecture on primitive roots would imply that $2$ should generate $\left(\mathbb Z/(p)\right)^\times= \left(\mathbb F_p\right)^\times$ for infinitely many $p$, and in fact, the proportion of primes for which 2 generates $\left(\mathbb F_p\right)^\times$ should asymptotically approach Artin's constant, which is approximately $0.3739558136\dots$ (see the survey \cite{M}).

\subsection{Random matrices with iid Rademacher $\pm 1$ entries}
While delocalization estimates for random polynomials with iid coefficients are fairly weak, we now consider random matrices with independent entries, for which much better delocalization bounds are known.  Indeed, we will use the following theorem from \cite{BVW} to bound the supremum in \eqref{eq:zerobnd}.

\begin{theorem}[Bourgain-Vu-Wood, Corollary~3.3 in \cite{BVW}] \label{BVWcor3.3}
Let $q$ be a constant such that $0 < q \le 1$ and let $S \subset \mathbb C$ be a set with cardinality $\abs S = O(1)$.  If $\mat{M}_n$ is an $n$ by $n$ matrix with independent random entries taking values in $S$ such that for any entry $x_{ij}$, we have $\max_{s \in S}\Pr(x_{ij} = s) \le q$, then $$\Pr(\mat M_n \mbox{ is singular}) \le \left( \sqrt q + o(1) \right)^n.$$
Furthermore, by inspecting the proof one can see that the $o(1)$ error term depends only on $q$ and the cardinality of the set $S$, and not on the values in the set $S$.
\end{theorem}

In \cite{BVW}, it was shown using the above result that an iid random Rademacher $\pm 1$ matrix (i.e., where each entry is $+1$ or $-1$ independently with probability $1/2$) is very unlikely to have a rational eigenvalue.  Our result below extends this fact by showing that, for any $\epsilon >0$, an eigenvalue that is algebraic with degree at most $n^{1/2-\epsilon}$ (which includes all rational numbers) is similarly unlikely.  Our approach here does not extend to algebraic degree $\sqrt n$ or larger; however, in analogy with Hilbert's Irreducibility Theorem and related results described in the introduction above, it seems likely that the characteristic polynomial of an iid random Rademacher $\pm 1$ matrix is in fact irreducible with high probability, which would imply that the matrix has no algebraic roots of degree less than $n$ (see \cite{BBBSWW} for supporting data).

\begin{theorem}\label{thm:iidpm1}
Let $\epsilon > 0$ be a constant, and let $\mat M_n$ be an $n$ by $n$ matrix where each entry takes the value $+1$ or $-1$ independently with probability $1/2$.  Then, the probability that $\mat M_n$ has an eigenvalue that is an algebraic number with degree at most $n^{1/2-\epsilon}$ is bounded above by $\displaystyle \left(\frac1{\sqrt 2} +o(1) \right)^n$.
\end{theorem}

\begin{proof}
Let $f$ be the characteristic polynomial of $\mat M_n$, so that the eigenvalues of $\mat M_n$ are the roots of $f$, all eigenvalues of $\mat M_n$ have absolute value at most $n$ with with probability 1 by an elementary bound.  (In fact, the eigenvalues of $M_n$ are all less than $O(\sqrt n)$ with exponentially high probability using, for example, \cite[Proposition~2.4]{RVnonasy}; we will not need such a refined bound here.)

Let $\Omega:=\{ z \in \mathbb C : \abs z \le n\}$.
Using Theorem~\ref{BVWcor3.3} above, we have for any $z\in \mathbb C$
that
\begin{equation}\label{eq:iidsingbd}
\Pr(f(z) = 0) = \Pr(\mat M_n - z\mat I_n \mbox{ is singular}) \le \left(\frac1{\sqrt 2} + o(1)\right)^n,
\end{equation}
where the $o(1)$ error is uniform for all $z \in \mathbb C$ (this follows using the facts that $\{1,-1, 1-z, -1-z\}$ is the set of values that can appear in $\mat M_n -z\mat I_n$ and that the cardinality of this set and the value of $q=1/2$ are the same for any $z \in \mathbb C$).
Thus,
$$\sup_{z \in \Omega} \Pr(f(z)=0) 
\le \left(\frac1{\sqrt 2} + o(1)\right)^n. $$
We now apply Lemma~\ref{lem:collected-bounds}(ii) and Theorem~\ref{thm:main} to complete the proof.
\end{proof}

\subsection{Random symmetric matrices}

In \cite{Vsym}, Vershynin proves a general result for real symmetric random matrices bounding the singularity probability, quantifying the smallest singular value, and showing that the spectrum is delocalized with the optimal scale.  Here, we will use the following special case showing only pointwise delocalization to illustrate an application of Theorem~\ref{thm:main}.  

\begin{theorem}[Vershynin, following from Theorem~1.2 in \cite{Vsym}] \label{thm:Vsym}
Let $B>0$ be a real constant and let $\mat M_n$ be a real symmetric $n$ by $n$ matrix whose entries $x_{ij}$ on and above the diagonal (so for $i\le j$) are iid random variables
with mean zero and unit variance 
satisfying $\abs{x_{ij}} \le B$. 
Then, there exists an absolute constant $c >0$ (depending only on $B$) such that, for every $r \in \mathbb R$,
\begin{equation} \label{eq:symbound}
	\Pr(r \mbox{ is an eigenvalue of } \mat M_n) \le 2e^{-n^c}.
\end{equation}
\end{theorem}
It is natural to only consider real numbers $r$ in \eqref{eq:symbound} since real symmetric matrices have all real eigenvalues.  
Also, the constant $c$ appearing in Theorem \ref{thm:Vsym} is typically less than one and may be much smaller.  

The more general version of the above result proven by Vershynin \cite[Theorem~1.2]{Vsym} applies to real symmetric matrices with entries having subgaussian tails (see \cite{Vintro} for why bounded implies subgaussian), and the bound we will prove on the probability of having low-degree algebraic numbers as eigenvalues (Theorem~\ref{thm:realsym} below) extends to this setting.

\begin{theorem}\label{thm:realsym}
Let $B>0$ be a real constant, let $c'>0$ be an absolute constant satisfying $c'< c/2$, where $c<1$ is the absolute constant from Theorem~\ref{thm:Vsym} (which depends only on $B$), and let $\mat M_n$ be an $n$ by $n$ real symmetric matrix whose entries on and above the diagonal are iid integer-valued random variables which are bounded in absolute value by $B$.  Then the probability that $\mat M_n$ has an eigenvalue that is algebraic of degree at most $n^{c'}$ is bounded above by
$e^{-n^c/2}$ for all sufficiently large $n$.
\end{theorem}

\begin{proof}[Proof of Theorem \ref{thm:realsym}]
Let $f$ be the characteristic polynomial of $\mat M_n$, so that the eigenvalues of $\mat M_n$ are the roots of $f$, and note that by \cite[Lemma~2.3]{Vsym}, all eigenvalues of $\mat M_n$ have absolute value at most $C\sqrt n$ with probability at least $1-2e^{-n}$ for some constant $C$ (depending only on $B$).

Let $\Omega:= \{r \in \mathbb R: \abs r \le C\sqrt n\}$.  Since $\mat M_n$ is a real symmetric matrix, the eigenvalues of $\mat M_n$ are all real.  
Moreover, Theorem~\ref{thm:Vsym} implies that $\sup_{r \in \Omega} \Pr(f(r)=0) \le 2e^{-n^c}$.  Thus, combining Theorem~\ref{thm:main}, Lemma~\ref{lem:collected-bounds}(iii), and \cite[Lemma~2.3]{Vsym}), we have that the probability that $f$ has an algebraic root of degree at most $n^{c'}$ is bounded above by $2\exp(-\pfrac 23 n^c) + 2e^{-n}$, which is at most $e^{-n^c/2}$ for all sufficiently large $n$.
%
\end{proof}

\subsection{Elliptical random matrices}

Elliptical random matrices interpolate between iid random matrices and random symmetric matrices.  In an elliptical random matrix, all the entries are independent with the exception that the $(i,j)$-entry may depend on the $(j,i)$-entry, and one also requires that the correlation between the $(i,j)$-entry and the $(j,i)$-entry is a constant $\rho$ for all $i\ne j$.  Thus, if the matrix has iid entries, then $\rho=0$, and if $\rho=1$, the matrix is symmetric.  There are results showing that the limiting distribution of the eigenvalues also interpolates between the limiting distributions for iid random matrices and for symmetric random matrices; in particular, for $-1 < \rho < 1$, the limiting eigenvalue distribution (suitably scaled) is an ellipse with eccentricity $\sqrt{1-\frac{(1-\rho)^2}{(1+\rho)^2}}$; see 
Nguyen and O'Rourke \cite{NO} and Naumov \cite{Nau}.

To apply Theroem~\ref{thm:main}, we will use a result due to Nguyen and O'Rourke \cite{NO} bounding the smallest singular value, and we will focus on the special case of $\pm 1$ elliptical random matrices for simplicity.  Let $\mat M_{n,\rho}$ be an elliptical random matrix with covariance parameter $-1<\rho<1$ with entries $x_{ij}$ defined as follows:
let $\{x_{i,j}: i \le j\} \cup \{\xi_{i,j}: i> j\}$ be a collection of independent random variables, where $\Pr(x_{i,j} = 1) = \Pr (x_{i,j} =-1)=1/2$ for $i\le j$ and where  
$$
\xi_{i,j} := 
\begin{cases}
1 & \mbox{ with probability }  (1+\rho)/2\\
-1 & \mbox{ with probability } (1-\rho)/2,
\end{cases}
$$
for $i > j$.  Then let $x_{i,j} := x_{j,i}\xi_{i,j}$ whenever $i > j$.  Define
$$ \mat M_{n, \rho} :=
\left(
\begin{matrix}
x_{1,1} & x_{1,2} & x_{1,3} &\dots  & x_{1,n} \\
x_{1,2}\xi_{2,1} & x_{2,2} & x_{2,3}& \dots & x_{2,n}\\
x_{1,3}\xi_{3,1} & x_{3,2}\xi_{3,2} & x_{3,3}& \dots& x_{3,n}\\
\vdots & \vdots & \ddots & \ddots &  \vdots \\
x_{1,n} \xi_{n,1} & x_{2,n} \xi_{n,2} & \dots & x_{n-1,n} \xi_{n,n-1}  & x_{n,n}
\end{matrix}
\right),
$$
and note that each entry takes the values $+1$ or $-1$ with equal probability.  We will call $\mat M_{n,\rho}$ a \emph{Rademacher elliptical random matrix with parameter $\rho$}.  

\begin{theorem}[Nguyen-O'Rourke, following from Theorem~1.9 in \cite{NO}]\label{thm:NO}
Let $\mat M_{n,\rho}$ be an $n$ by $n$ Rademacher elliptical random matrix with parameter $-1 < \rho < 1$, and let $B'>0$ be a constant.  Then, for all sufficiently large $n$ (depending only on $B'$ and $\rho$), we have that
$$\sup_{z\in \mathbb C,\  \abs z \le n} \Pr(z \mbox{ is an eigenvalue of }\mat M_{n,\rho}) \le n^{-B'}.
$$ 
\end{theorem}

We can now apply Theorem~\ref{thm:main} to get the following result.

\begin{theorem} \label{thm:elliptical}
Let $\mat M_{n,\rho}$ be an $n$ by $n$ Rademacher elliptical random matrix with parameter $-1 < \rho < 1$, and let $B >0$ and $K \geq 1$ be constants.  Then, for all sufficiently large $n$ (depending only on $B$, $\rho$, and $K$), the probability that the matrix $\mat M_{n,\rho}$ has an eigenvalue that is algebraic of degree at most $K$ is bounded above by $n^{-B}$.
\end{theorem}

\begin{proof}
Let $f$ be the characteristic polynomial of $\mat M_{n,\rho}$. 
All eigenvalues of $\mat M_{n,\rho}$ have absolute value at most $n$ with probability $1$ by an elementary bound.
Let $\Omega:= \{z \in \mathbb C: \abs z \le n\}$, and note that Theorem~\ref{thm:NO} lets us take $p=n^{-B'}$ for any constant $B'>0$ in \eqref{eq:zerobnd}.  Thus, we may apply Theorem~\ref{thm:main} and Lemma~\ref{lem:collected-bounds}(iv) (with $M=n$ and $m=1$) to complete the proof.
%
\end{proof}

\subsection{Product matrices}

We now show how Theorem~\ref{thm:main} can be applied to products of independent random matrices.  We begin with the following result from \cite{ORSV}.  

\begin{theorem}[O'Rourke-Renfrew-Shoshnikov-Vu, \cite{ORSV} Theorem~5.2 ] \label{thm:ORSV}
Let $m \geq 1$ and $B', \gamma > 0$ be constants.  Let $\mat M_n^{(1)}, \ldots, \mat M_n^{(m)}$ be independent $n$ by $n$ matrices in which each entry takes the value $+1$ or $-1$ independently with probability $1/2$.  Define the product
$$ \mat M_n := \mat M_n^{(1)} \cdots \mat M_n^{(m)}. $$
Then, for all sufficiently large $n$ (depending only on $m$, $B'$, and $\gamma$), we have
$$ \sup_{z \in \mathbb{C}, \  |z| \leq n^{\gamma}} \Prob \left( z \mbox{ is an eigenvalue of } \mat M_n \right) \leq n^{-B'}. $$
\end{theorem}

We can now apply Theorem~\ref{thm:main} to get the following result.

\begin{theorem} \label{thm:prod}
Let $K, m \geq 1$ and $B > 0$ be constants.  Let $\mat M_n^{(1)}, \ldots, \mat M_n^{(m)}$ be independent $n$ by $n$ matrices in which each entry takes the value $+1$ or $-1$ independently with probability $1/2$.  Then, for all sufficiently large $n$ (depending only on $m$, $K$, and $B$), the probability that the matrix 
$$ \mat M_n := \mat M_n^{(1)} \cdots \mat M_n^{(m)} $$
has an eigenvalue that is algebraic of degree at most $K$ is bounded above by $n^{-B}$.  
\end{theorem}

\begin{proof}
Let $f$ be the characteristic polynomial of $\mat M_{n}$, 
and note that all eigenvalues of $\mat M_{n}$ have absolute value at most $n^m$ with probability 1 by an elementary bound.
Let $\Omega:= \{z \in \mathbb C: \abs z \le n^m\}$, and note that by Theorem~\ref{thm:ORSV} we can take $p=n^{-B'}$ for any constant $B'>0$ in \eqref{eq:zerobnd}. Thus, we may apply Theorem~\ref{thm:main} and Lemma~\ref{lem:collected-bounds}(iv) (with $M=n^m$) to complete the proof.
%
\end{proof}

More generally, Theorem \ref{thm:ORSV} can be extended to products of elliptical random matrices which satisfy a number of constraints (see \cite[Theorem 5.2]{ORSV} for details).  This leads naturally to a version of Theorem \ref{thm:prod} for the product of $m$ independent Rademacher elliptical random matrices with parameters $\rho_1, \ldots, \rho_m$ satisfying $-1 < \rho_i < 1$.  


\subsection{Erd\H os--R\'enyi random graphs}\label{ss:simplegraphs}
We now consider Erd\H os--R\'enyi random graphs on $n$ vertices, where each edge is present independently at random with a constant probability $p$ satisfying $0<p<1$.  We denote such a graph by $G(n,p)$ and observe that the graph can be defined by its adjacency matrix $\mat A_n$, which is a real symmetric matrix with entry $(i,j)$ equal to 1 if there is an edge between vertices $i$ and $j$, and entry equal to zero otherwise.

In the Erd\H{o}s--R\'enyi model, the independence among edges means that all entries in the strict upper triangle of $\mat A_n$ are also independent.  Thus, the following result due to Nguyen \cite{Nsym} is applicable.  

\begin{theorem}[Nguyen, following from Theorem 1.4 in \cite{Nsym}] \label{thm:Nsym}
Let $0 < p < 1$ and $B' > 0$ be constants, and let $\mat A_n$ be the adjacency matrix of $G(n,p)$.  Then, for $n$ sufficiently large (depending only on $p$ and $B'$), 
$$ \sup_{z \in \mathbb{C}, \ |z| \leq n} \Prob( z \mbox{ is an eigenvalue of } \mat A_n ) \leq n^{-B'}. $$
\end{theorem}

By following the proof of Theorem \ref{thm:elliptical} and applying Theorem \ref{thm:Nsym} in place of Theorem~\ref{thm:NO}, we find that for any $K \geq 1$ and $B > 0$, the probability that $\mat A_n$ has an eigenvalue that is algebraic of degree at most $K$ is bounded above $n^{-B}$ for $n$ sufficiently large (depending only on $K, B$ and $p$).  We state this result explicitly in Section \ref{sec:control} (see Theorem~\ref{thm:ER}).  The result is also true when the diagonal entries of $\mat A_n$ are allowed to be one (this corresponds to the case where loops are allowed in the graph).  

\subsection{Directed random graphs}
In the case of directed random graphs where directed edges (including loops) are included independently at random with probability $p$, where $0<p<1$ is a constant, the adjacency matrix $\mat M_n$ is an $n$ by $n$ matrix with entries independently equal to $1$ with probability $p$, and otherwise the entries are zero.   In this case, Theorem~\ref{BVWcor3.3} applies with $q:=\max\{ p, 1-p\}$, and thus, following the proof of Theorem~\ref{thm:iidpm1}, proves that for any $\epsilon >0$, the probability that $\mat M_n$ has an eigenvalue that is an algebraic number with degree at most $n^{1/2-\epsilon}$ is bounded above by $\left(\sqrt{q} + o(1)\right)^n$.

\subsection{Directed random graphs with fixed outdegrees}

Let $s$ be a positive integer, and let $x \in \{0,1\}^n$ be a random binary vector uniformly chosen from among all binary vectors containing exactly $s$ ones.  If $\mat M_n$ is the $n \times n$ matrix whose rows are iid copies of the vector $x$, then $\mat M_n$ can be viewed as the adjacency matrix of a random directed graph on $n$ vertices (where loops are allowed) such that each vertex has outdegree $s$.  In this case, $\mat M_n$ always has $s$ as an eigenvalue (with the corresponding eigenvector being the all-ones vector), and hence not every eigenvalue of $\mat M_n$ can be of high algebraic degree.  Using Theorem~\ref{thm:main}, we show that, besides this trivial eigenvalue, the other eigenvalues cannot be low-degree algebraic numbers.  

\begin{theorem} \label{thm:outdegree}
Let $0 < \eps \leq 1$, $K \geq 1$, and $B > 0$ be a constants, and let $x \in \{0,1\}^n$ be a random binary vector uniformly chosen from among all binary vectors containing exactly $s$ ones for some $s$ satisfying $|s - n/2| \leq (1-\eps)n/2$.  If $\mat M_n$ is a random $n$ by $n$ matrix whose rows are iid copies of the vector $x$, then, for all sufficiently large $n$ (depending only on $\eps$, $K$, and $B$), the probability that one of the non-trivial eigenvalues of the matrix $\mat M_n$ is algebraic of degree at most $K$ is bounded above by $n^{-B}$.  
\end{theorem}

\begin{proof}
The proof of Theorem \ref{thm:outdegree} follows closely the proof of Theorem \ref{thm:elliptical}, where instead of using Theorem \ref{thm:NO} we apply Theorem \ref{thm:singoutdegree} below.  The main difference comes from the fact that we must now deal with the trivial eigenvalue at $s$.  

Let $f$ be the characteristic polynomial of $\mat M_{n}$, 
and note that all eigenvalues of $\mat M_{n}$ have absolute value at most $n$ with probability 1 by an elementary bound.
Let $\Omega:= \{z \in \mathbb C: \abs z \le n, z \neq s \}$, and note that by Theorem~\ref{thm:singoutdegree} below, we may take $p=n^{-B'}$ for any constant $B'>0$ in \eqref{eq:zerobnd}. Thus, we may apply Theorem~\ref{thm:main} and Lemma~\ref{lem:collected-bounds}(iv) (with $M=n$ and $m=1$) to complete the proof.
%
\end{proof}

It remains to verify the following bound.  

\begin{theorem} \label{thm:singoutdegree}
Let $0 < \eps \leq 1$ and $B > 0$ be a constants, and let $x \in \{0,1\}^n$ be a random binary vector uniformly chosen from among all binary vectors containing exactly $s$ ones for some $s$ satisfying $|s - n/2| \leq (1-\eps)n/2$.  If $\mat M_n$ is a random $n$ by $n$ matrix whose rows are iid copies of the vector $x$, then, for $n$ sufficiently large (depending only on $\eps$ and $B$), 
\begin{equation} \label{eq:odbnd1}
	\sup_{z \in \mathbb{C}, \ z \neq s} \Prob(z \mbox{ is an eigenvalue of } \mat M_n) \leq n^{-B}  
\end{equation}
and
\begin{equation} \label{eq:odbnd2}
	\Prob( s \mbox{ is an eigenvalue of } \mat M_n \mbox{ with algebraic multiplicity at least } 2) \leq n^{-B}. 
\end{equation}
\end{theorem}
\begin{proof}
The proof follows the arguments given by Nguyen and Vu in \cite{NV}.  We begin with the bound in \eqref{eq:odbnd1}.  Let $\Omega := \{z \in \mathbb{C} : |z| \leq n, z \neq s\}$.  Since, with probability $1$, all eigenvalues of $\mat M_n$ are contained in the disk $\{z \in \mathbb{C} : |z| \leq n\}$, it suffices to show
$$ \sup_{z \in \Omega} \Prob(z \mbox{ is an eigenvalue of } \mat M_n) \leq n^{-B} $$
for $n$ sufficiently large.  Define the matrix
$ \mat X_n := 2 \mat M_n - \mat J_n, $
where $\mat J_n$ is the $n \times n$ all-ones matrix.  In particular, $\mat X_n$ is an $n \times n$ random matrix with $+1$ and $-1$ entries whose rows are independent with row sum $2s - n$, where $|2s-n| \leq (1-\eps)n$.  Such matrices were explicitly studied in \cite{NV}, and the estimate below follows from \cite[Theorem 2.8]{NV}.  Let $\mat M_{n-1}$ be the $(n-1) \times (n-1)$ submatrix of $\mat M_n$ formed from $\mat M_n$ by removing the last row and column.  Similarly, let
$ \mat X_{n-1} := 2 \mat M_{n-1} - \mat J_{n-1}. $
Then, for any deterministic matrix $\mat F$ satisfying $\|\mat F\| \leq n^2$, \cite[Theorem 2.8]{NV} implies that
\begin{equation} \label{eq:supzNV}
	\sup_{z \in \mathbb{C}, \ |z| \leq 2n} \Prob(z \mbox{ is an eigenvalue of } \mat X_{n-1} + \mat F) \leq n^{-B} 
\end{equation}
for all $n$ sufficiently large (depending only on $\eps$ and $B$). 

The advantage of working with $\mat M_{n-1}$ is that it does not have a trivial eigenvalue at $s$.  Thus, we will reduce to the case where the bound in \eqref{eq:supzNV} is relevant.  Let $m_{ij}$ denote the $(i,j)$-entry of $\mat M_n$.  Define $\mat M := \mat M_n - z \mat I_n$.  Then $\det (\mat M) = \det (\mat M')$, where $\mat M'$ is obtained from $\mat M$ by adding the first $n-1$ columns to the last column.  Since each entry of the last column of $\mat M'$ takes the value $s - z$, $\det (\mat M') = (s-z) \det (\mat M'')$, where $\mat M''$ is obtained from $\mat M$ by replacing each entry in the last column by $1$, i.e., 
$$ \mat M'' := \begin{bmatrix} m_{1,1} - z & m_{1,2} & \dots & m_{1, n-1} & 1 \\ \vdots & \vdots & \ddots & \vdots & \vdots \\ m_{n-1,1} & m_{n-1, 2} & \dots & m_{n-1,n-1} - z & 1 \\ m_{n,1} & m_{n,2} & \dots & m_{n, n-1} & 1 \end{bmatrix}. $$
Since $s \not\in \Omega$, it now suffices to show
\begin{equation} \label{eq:odsufftoshow}
	\sup_{z \in \mathbb{C}, \ |z| \leq n} \Prob(\det(\mat M'') = 0) \leq n^{-B} 
\end{equation}
for $n$ sufficiently large.  Additionally, as $\det(\mat M_n - z \mat I) = (s - z) \det \mat (\mat M'')$, the bound in \eqref{eq:odsufftoshow} would also imply \eqref{eq:odbnd2}.  

By subtracting the last row of $\mat M''$ from each of the previous $n-1$ rows, it follows that
$ \det( \mat M'') = \det (\mat M_{n-1} - \mat Q_{n-1} - z \mat I_{n-1}), $
where $\mat Q_{n-1}$ is an $(n-1) \times (n-1)$ rank-one matrix whose rows are each given by $(m_{n,1}, \ldots, m_{n,n-1})$.  Since the entries $m_{n,1}, \ldots, m_{n,n-1}$ are independent of the entries in $\mat M_{n-1}$, we condition on $\mat Q_{n-1}$ and now treat this matrix as deterministic.  Observe that
$ \det (\mat M_{n-1} - \mat Q_{n-1} - z \mat I_{n-1}) = 0 $
if and only if $2z$ is an eigenvalue of
$$ 2 \mat M_{n-1} - 2\mat Q_{n-1} = \mat X_{n-1} - 2 \mat Q_{n-1} + \mat J_{n-1} =: \mat X_{n-1} + \mat F. $$
By an elementary bound, 
$$ \|\mat F\| \leq 2 \| \mat Q_{n-1} \| + \|\mat J_{n-1} \| \leq 3n \leq n^2 $$
for $n \geq 3$.  Therefore, we conclude from \eqref{eq:supzNV} that
$$ \sup_{z \in \mathbb{C}, \  |z| \leq n} \Prob( \det (\mat M_{n-1} - \mat Q_{n-1} - z \mat I_{n-1}) = 0 ) \leq n^{-B} $$
for $n$ sufficiently large, and the proof is complete.  
\end{proof}

\subsection{Other models}
In the previous subsections, we focused on random polynomials models for which good pointwise delocalization bounds are known, especially characteristic polynomials of random matrices.  For example, the approach above  also works for sparse random matrices, using Tao and Vu's \cite[Theorem~2.9]{TVsparse} to show pointwise delocalization.

However, there are many other models of random matrices one could consider.  For instance, sample covariance matrices arise in many applications and are well-studied in the random matrix theory literature.  Yet, the authors are not aware of delocalization bounds of the form required for Theorem~\ref{thm:main}.  Another interesting model is random matrices with exchangeable entries.  While Adamczak, Chafa\"{i}, and Wolff \cite{ACW} have obtained some delocalization bounds for such matrices, the bounds are not strong enough to use with Theorem~\ref{thm:main}.   Some delocalization bounds have been proven by Cook \cite{C1, C2} for the adjacency matrix and signed adjacency matrix of such graphs, and it would be interesting to see if 
strong enough bounds could be proven to combine with Theorem~\ref{thm:main}.

\section{Motivation: Random graphs and controllability} \label{sec:control}

As discussed above, our main results are motivated by the question of whether a random polynomial with integer coefficients is likely to be irreducible.  In particular, we have focused on characteristic polynomials of matrices, and it is natural to ask whether such models have applications.  

In this section, we motivate these models by discussing graphs and their adjacency matrices.  Unsurprisingly, certain properties of a graph can be deduced from the characteristic polynomial of its adjacency matrix.  Specifically, we focus on the property of symmetry, which in turn is related to controllability properties of a certain linear system formed from the graph.  In this section, we provide a brief introduction to linear control theory, random graphs, and their connection with our main results.  The uninterested reader can safely skip this section.  

\subsection{Linear control theory}
Generally speaking, linear control theory is concerned with controlling linear systems, so the output (or solution) of the system follows a desired path.  In what follows, we shall consider a very specific linear system formed from a matrix $\mat A$ and a vector $b$.  

Let $\mat A$ be an $n \times n$ matrix with real entries, and let $b$ be a vector in $\mathbb{R}^n$.  Then the \emph{continuous time-invariant control system} formed from the pair $(\mat A,b)$ is defined by the equation
\begin{equation} \label{eq:linear}
	\dot{x}(t) = \mat A x(t) + u(t) b, 
\end{equation}
where $u: [t_0, t_1] \to \mathbb{R}$ is called the \emph{control} and $x: [t_0, t_1] \to \mathbb{R}^n$ is called the \emph{state} of the system.  Here, $\dot{x}$ denotes the time derivative of $x$.  We typically view $\mat A, b$, and $u$ as given values and $x$ as the solution to \eqref{eq:linear}.  In particular, given $\mat A, b$, an initial value $x(t_0)$, and sufficiently smooth $u$, the state $x$ is uniquely determined by \eqref{eq:linear}.

We want to consider the general property of being able to ``steer'' such a system from any given state to any other by a suitable choice of the control function $u$.  This ability to ``steer'' the system is what we will mean by the term controllability.  

\begin{definition}[Complete controllability]
Let $\mat A$ be an $n \times n$ matrix with real entries, and let $b$ be a vector in $\mathbb{R}^n$.  We say the pair $(\mat A,b)$ is \emph{completely controllable} if, for any $t_0$, any initial state $x(t_0) = x_0$, and any given final state $x_f$, there exists $t_1 > t_0$ and a piecewise continuous control $u:[t_0, t_1] \to \mathbb{R}$ such that the solution (state) of \eqref{eq:linear} satisfies $x(t_1) = x_f$.  
\end{definition}

\begin{remark}
The qualifying term ``completely'' implies that the definition holds for all $x_0$ and $x_f$.  In general, several other types of controllability can also be defined.   
\end{remark}

The basic problem that now arises is to describe exactly which pairs $(\mat A,b)$ are completely controllable.  Kalman's rank condition \cite{K2, K3, Klec, K} gives a general algebraic criterion.  

\begin{theorem}[Kalman \cite{K}] \label{thm:kalman}
Let $\mat A$ be an $n \times n$ matrix with real entries, and let $b$ be a vector in $\mathbb{R}^n$.  The pair $(\mat A,b)$ is completely controllable if and only if the \emph{controllability matrix}
\begin{equation} \label{eq:kalmanmatrix}
	\begin{bmatrix} \mat b & \mat Ab & \mat A^2 b & \cdots & \mat A^{n-1} b \end{bmatrix} 
\end{equation}
has full rank (that is, rank $n$).  Here, the matrix in \eqref{eq:kalmanmatrix} is the $n \times n$ matrix with columns $b$, $\mat Ab$, $\mat A^2b$, \ldots, $\mat A^{n-1}b$.  
\end{theorem}

Theorem \ref{thm:kalman} is so convenient that this rank condition is often taken as the definition of controllability.  In fact, from this point forward, we will no longer consider the linear system in \eqref{eq:linear}.  Instead, we will only focus on the controllability matrix \eqref{eq:kalmanmatrix}.  To this end, we make the following definition.

\begin{definition}[Controllability] 
Let $\mat A$ be an $n \times n$ matrix with real entries, and let $b$ be a vector in $\mathbb{R}^n$.  We say the pair $(\mat A,b)$ is \emph{controllable} if the controllability matrix, defined in \eqref{eq:kalmanmatrix}, has rank $n$.  If $(\mat A,b)$ is not controllable, we say the pair is \emph{uncontrollable}.  
\end{definition}

\begin{remark}
In view of Theorem \ref{thm:kalman}, controllability and complete controllability are equivalent.  We drop the qualifying term ``complete'' as this is the only type of controllability we will consider.  
\end{remark}

\subsection{Controllable subsets in graphs}
Let $G$ be a simple graph on the vertex set $[n] := \{1, \ldots, n\}$ with adjacency matrix $\mat A$, i.e., $\mat A$ is a real symmetric matrix with entry $(i,j)$ equal to $1$ if there is an edge between vertices $i$ and $j$, and the entry is equal to zero otherwise.  In this section, we focus on the controllability of $(\mat A, b)$.  Of particular importance is the case when $b \in \{0,1\}^n$ is a binary vector.  Indeed, in this case, $b$ can be viewed as the characteristic vector of some subset of the vertex set $[n]$.  We make the following definitions.  We say the simple graph $G$ on $n$ vertices is \emph{controllable} if $(\mat A, \1)$ is controllable, where $\mat A$ is the adjacency matrix of $G$ and $\1$ is the all-ones vector in $\mathbb{R}^n$.  Additionally, we say $G$ is \emph{minimally controllable} if $(\mat A, e_i)$ is controllable for every $1 \leq i \leq n$, where $e_1, \ldots, e_n$ is the standard basis of $\mathbb{R}^n$.  

Studying the controllability properties of large scale graphs and networks has become an important and challenging task in control theory with several real-world applications.  For instance, one of the emerging applications of network controllability is the control of neural networks inside the brain and its relation to behavioral regulation \cite{FHEPL,GPetal}.  In this application,\footnote{Both of the applications mentioned here typically involve studying matrices other than the adjacency matrix of the underlying graph.  For simplicity, we will only consider the adjacency matrix in this paper.} the neural network in the brain is modeled as a graph with each vertex representing a neuron or region in the brain.  

Another application involves studying social influence.  Indeed, with the prevalence of online social networks, social influence is now a highly studied topic due, in part, to its use in categorizing efficient mechanisms for the spread of information as well as identification of susceptible members of society \cite{AW, BFJK}.  In this application, the graph in question is the social network, and the characteristic vector $b$ can be viewed as identifying the ``leaders'' in the network who try to control the other individuals.  

We recall the following elementary definitions.  \emph{Isomorphisms} of simple graphs are bijections of the vertex sets preserving adjacency as well as non-adjacency.  \emph{Automorphisms} of the graph $G$ are $G \to G$ isomorphisms.  Clearly, the identity map is always an automorphism.  A graph is called \emph{asymmetric} if it has no non-trivial automorphisms.  

We now discuss some connections which exist between controllability, asymmetry, and the characteristic polynomial of the adjacency matrix.  

\begin{proposition}[Godsil, following from Lemma 1.1 in \cite{G}] \label{prop:godsil}
If the simple graph $G$ is controllable, then $G$ is asymmetric.  
\end{proposition}

Godsil, in \cite{G}, showed a connection between the characteristic polynomial of the adjacency matrix and controllability.  

\begin{theorem}[Godsil, Corollary 5.3 in \cite{G}] \label{thm:godsil}
Let $G$ be a simple graph with adjacency matrix $\mat A$.  If the characteristic polynomial of $\mat A$ is irreducible over the rationals, then $G$ is controllable and minimally controllable.  
\end{theorem} 

Putting these two results together, we recover the well-known implication (see, for example, \cite{CG}) that if $G$ is a simple graph with adjacency matrix $\mat A$ and the characteristic polynomial of $\mat A$ is irreducible over the rationals, then $G$ is asymmetric.

\subsection{Conjectures and results concerning random graphs}
Recall that $G(n,p)$ is the Erd\"os--R\'enyi random graph on the vertex set $[n]$ with edge density $p$.  That is, $G(n,p)$ is a simple graph on $n$ vertices (which we shall label as $1, \ldots, n$) such that each edge $\{i,j\}$ is in $G(n,p)$ with probability $p$, independent of other edges.  In the special case when $p=1/2$, one can view $G(n,1/2)$ as random graph selected uniformly among all $2^{\binom{n}{2}}$ simple graphs on $n$ vertices.  We let $\mat A_n$ be the zero-one adjacency matrix of $G(n,p)$.  

It was proven by P\'{o}lya \cite{P} and Erd\H os and R\'enyi \cite{ER} that $G(n,1/2)$ is asymmetric with probability $1 - \binom{n}{2} n^{-n-2}(1 + o(1))$; see \cite{Ba} and references therein for further details.  In other words, most simple graphs are asymmetric. In view of Proposition \ref{prop:godsil}, this gives an upper bound for the probability that $G(n,p)$ is controllable.  In terms of a lower bound, Godsil \cite{G} has recently conjectured that most simple graphs are controllable and minimally controllable.  

\begin{conjecture}[Godsil \cite{G}] \label{conj:godsil}
The probability that $G(n, 1/2)$ is controllable and minimally controllable approaches $1$ as $n \to \infty$.  
\end{conjecture}

One can view Conjecture \ref{conj:godsil} as stating that controllability (alternatively, minimal controllability) is a universal property of graphs.   Conjecture \ref{conj:godsil} was recently proven in \cite{OT, OT2}.  The proof relies on Kalman's rank condition (Theorem \ref{thm:kalman}) and one of its corollaries known as the Popov--Belevitch--Hautus (PBH) test (see \cite[Section 12.2]{Hlst} for details).  In particular, the proof given in \cite{OT, OT2} involves studying the additive structure of the eigenvectors of the random adjacency matrix $\mat A_n$.  

It has also been conjectured (and numerical evidence suggests) that the characteristic polynomial of the adjacency matrix of $G(n,1/2)$ is irreducible over the rationals with high probability.  The authors are not aware of any progress in proving this conjecture.  In view of Theorem \ref{thm:godsil}, though, this conjecture would imply Conjecture \ref{conj:godsil}.  Specifically, Theorem \ref{thm:godsil} hints at another approach to prove Conjecture \ref{conj:godsil}, which would be entirely different from the proofs given in \cite{OT, OT2}.  While the proofs in these previous works focused on the eigenvector structure, this new method only requires working with the eigenvalues (in particular, the characteristic polynomial) of $\mat A_n$.  If one could show, for instance, that, with high probability, $\mat A_n$ has no eigenvalues that are algebraic of degree at most $n/2$, then both conjectures would follow.  While our main results do not go so far, they do hint that this may indeed be the case.   Theorem \ref{thm:ER} below follows from Theorem \ref{thm:Nsym} and the reasoning in Subsection~\ref{ss:simplegraphs}.

\begin{theorem} \label{thm:ER}
Fix $0<p<1$, and let $B > 0$ and $K \geq 1$ be constants.  Let $G(n,p)$ be an Erd\H os--R\'enyi random graph on $n$ vertices with edge density $p$.  Then, for $n$ sufficiently large (depending on $B$, $K$, and $p$), the probability that the adjacency matrix $\mat A_n$ of $G(n,p)$ has an eigenvalue that is algebraic of degree at most $K$ is bounded above by $n^{-B}$.  
\end{theorem}

Note that Vershynin's result \cite[Theorem~1.2]{Vsym} (and also the special case stated in Theorem~\ref{thm:Vsym}) is not applicable here because then entries of the adjacency matrix do not have zero mean.  It would be interesting to see if \cite[Theorem~1.2]{Vsym} could be extended to the case where the entries had non-zero mean; such a result would directly improve the bound in Theorem~\ref{thm:ER} above, likely giving a result analogous to Theorem~\ref{thm:realsym}.

\section{Proof of Theorem~\ref{thm:main} and Lemma~\ref{lem:collected-bounds}} \label{sec:main}

We prove Theorem~\ref{thm:main} first and prove Lemma~\ref{lem:collected-bounds} at the end of this section.
We prove Theorem~\ref{thm:main} via a series of lemmata.  Some of the results in this section can also be found in the text \cite{DF} by Dummit and Foote; we provide proofs in certain cases for completeness.  

\begin{lemma} \label{lemma:minpoly}
Let $f$ be a polynomial with rational coefficients.  If $\lambda$ is a root of $f$, then the minimal polynomial of $\lambda$ divides $f$ over the rationals.  
\end{lemma}
\begin{proof}
Let $g$ denote the minimum polynomial of $\lambda$.  By definition of the minimum polynomial, this implies that $\deg(g) \leq \deg(f)$.  Hence, by the division algorithm,
$f(z) = h(z) g(z) + r(z),$
where $h$ and $r$ are polynomials with rational coefficients and $\deg(r) < \deg(g)$.  
Since $\lambda$ is a root of both $f$ and $g$, we have that $\lambda$ is a root of $r$.  However, since $\deg(r) < \deg(g)$ and $g$ is the minimum polynomial of $\lambda$, we must have that $r(z)=0$.  
\end{proof}

For the proof of the next lemma, we will need Gauss's lemma.

\begin{theorem}[Gauss's lemma; Proposition 5 on page 303 of \cite{DF}] \label{thm:gauss}
Let $f$ be a nonconstant polynomial with integer coefficients.  If $f$ is irreducible over the integers, then $f$ is irreducible over the rationals.  
\end{theorem}

\begin{lemma} \label{lemma:alginteger}
Let $f$ be a monic polynomial with integer coefficients.  If $\lambda$ is a root of $f$ with minimal polynomial $g$, then $g$ is a monic polynomial with integer coefficients and $\lambda$ is an algebraic integer.  
\end{lemma}
\begin{proof}
We begin by factoring $f$ over the integers into irreducible polynomials $f_j$ for $1 \leq j \leq \ell$:
$$ f(z) = \prod_{j=1}^\ell f_j(z). $$
It must be the case that each $f_j$ is monic.  Additionally, $\lambda$ must be a root of one of the $f_j$'s; without loss of generality, assume $\lambda$ is a root of $f_1$.  By Lemma \ref{lemma:minpoly}, $g$ divides $f_1$ over the rationals.  However, from Gauss's lemma (Theorem \ref{thm:gauss}), this implies (since $f_1$ is monic) that $g=f_1$.  We conclude that $g$ is a monic polynomial with integer coefficients, and by definition it follows that $\lambda$ is an algebraic integer.  
\end{proof}

Lemma \ref{lemma:counting} below is the main lemma we will need to prove Theorem~\ref{thm:main}.  Roughly speaking, Lemma \ref{lemma:counting} says that if $f$ is a monic polynomial with integer coefficients and bounded roots, then there are only a limited number of points in $\mathbb{C}$ that can be roots for $f$ that are algebraic with low degree.  

\begin{lemma}[Counting bound] \label{lemma:counting}
Let $M > 0$ and $1 \leq k \leq n$.  Then there exists a set $S \subset \mathbb{C}$ (depending only on $M$ and $k$) of algebraic integers with cardinality 
$$ |S| \leq k \prod_{j=1}^k \left(2 \binom{k}{j} M^j + 1 \right) $$
such that the following holds.  If $f$ is a monic polynomial of degree $n$ with integer coefficients whose roots are bounded in magnitude by $M$ and $\lambda$ is an root of $f$ that is algebraic of degree $k$, then $\lambda \in S$.  
\end{lemma}
\begin{proof}
For each $c_0, \ldots, c_{k-1} \in \mathbb{Z}$, we define the monic polynomial with coefficients $c_0, \ldots, c_{k-1}$ as 
$$ h_{c_0, \ldots, c_{k-1}}(z) := z^k + \sum_{j=1}^k c_{k-j} z^{k-j}. $$
Each such polynomial is a monic polynomial with integer coefficients, and hence the roots of any such polynomial are always algebraic integers.  Define the index set 
$$ T:= \left\{ (c_0, \ldots, c_{k-1}) \in \mathbb{Z}^k : |c_{k-j}| \leq \binom{k}{j} M^j \text{ for } j = 1, \ldots, k \right\}. $$
By construction, 
$$ |T| \leq \prod_{j=1}^k \left( 2 \binom{k}{j} M^j + 1 \right). $$
We now define the set $S$  as the collection of roots of all polynomials $h_{c_0, \ldots, c_{k-1}}$ whose coefficients $(c_0, \ldots, c_{k-1}) \in T$.  In other words, recalling \eqref{def:Lambda}, 
$$S:= \bigcup_{(c_0, \ldots, c_{k-1}) \in T} \Lambda(h_{c_0, \ldots, c_{k-1}}).$$ Since each polynomial $h_{c_0, \ldots, c_{k-1}}$ has at most $k$ distinct roots, it follows that 

\begin{equation}\label{eq:Sbound}
|S| \leq k |T| \leq k \prod_{j=1}^k \left( 2 \binom{k}{j} M^j + 1 \right). 
\end{equation}

We now claim that $S$ satisfies the conclusion of the lemma.  Indeed, let $f$ be a monic polynomial of degree $n$ with integer coefficients whose roots are bounded in magnitude by $M$.  Let $\lambda_1, \ldots, \lambda_n$ be the roots of $f$, and suppose $\lambda_1$ is an algebraic root of degree $k$.  It follows from Lemmas \ref{lemma:minpoly} and \ref{lemma:alginteger}, that the minimal polynomial of $\lambda_1$, say $g$, is a monic polynomial with integer coefficients which divides $f$.  This implies that the roots of $g$ are also roots of $f$.  (Clearly, $\lambda_1$ is a root of both $f$ and $g$.)  Without loss of generality, assume $\lambda_1, \ldots, \lambda_k$ are the roots of $g$.  Then
$$ g(z) = (z - \lambda_1) \cdots (z - \lambda_k) = z^k + \sum_{j=1}^k d_{k-j} z^{k-j}, $$
where $d_{k-j} := (-1)^j \sum_{1 \leq i_1 < \cdots < i_j \leq k} \lambda_{i_1} \cdots \lambda_{i_j}. $
As noted above, each $d_{k-j} \in \mathbb{Z}$.  In addition, since each root of $f$ is bounded in magnitude by $M$, it follows that
$ |d_{k-j} | \leq \binom{k}{j} M^j. $
This implies that $(d_0, \ldots, d_{k-1}) \in T$.  Therefore, we conclude that the roots of $g$ are contained in $S$.  
\end{proof}

With Lemma \ref{lemma:counting} in hand, we are now ready to prove Theorem~\ref{thm:main}.  The main idea is  simple: If $f$ does have an algebraic root of degree $k$, then Lemma \ref{lemma:counting} shows it must be contained in the set $S$, which has small cardinality.  We can then show that each of the points in $S$ is unlikely to be a root of $f$ using the bound in \eqref{eq:zerobnd}.  

\begin{proof}[Proof of Theorem~\ref{thm:main}]
Let $f$ be a random monic polynomial with integer coefficients.  
Let $S \subset \mathbb{C}$ be the set of algebraic integers from Lemma \ref{lemma:counting}.  In particular, $S$ is a deterministic set which only depends on $M$ and $k$, and $S$ has cardinality 
\begin{equation} \label{eq:Scard}
	|S| \leq k \prod_{j=1}^k \left( 2 \binom{k}{j} M^j + 1 \right). 
\end{equation}
Let $\mathcal{B}_{f,M}$ be the event that all roots $z$ of $f$ satisfy $\abs z \le M$. If $f$ has an algebraic root of degree $k$ in $\Omega$ and $\mathcal{B}_{f,M}$ holds, then Lemma \ref{lemma:counting} implies that this root must be in $S \cap \Omega$.  Hence, by the union bound, we obtain 
\begin{align*}
	\Prob &\left( f \text{ has an algebraic root of degree } k \text{ in } \Omega \right) \\
	&\quad\leq \Prob \left(\{ \mbox{there exists } w \in S \cap \Omega \mbox{ such that } f(w) = 0\} \cap \mathcal{B}_{f,M} \right) + \Prob(\overline{\mathcal{B}_{f,M}})\\
	&\quad\leq \left(\sum_{w \in S \cap \Omega} \Prob \left( f(w) = 0 \right)\right) + \Prob(\overline{\mathcal{B}_{f,M}}) \\
&\quad\leq p |S| + \Prob(\overline{\mathcal{B}_{f,M}}).
\end{align*}
The conclusion now follows from the cardinality bound given in \eqref{eq:Scard} combined with Lemma~\ref{prop:bounds} (based on Stirling's approximation) below.
\end{proof}

\begin{lemma}[Some useful bounds] \label{prop:bounds}
For $M \geq 1$ and $k \ge 2$, 
\begin{align} \label{eq:prodbnd}
M^{(k^2+k)/2}e^{(k^2 -k\log(k))/2} \leq	\prod_{j=1}^k \left( 2 \binom{k}{j} M^j + 1 \right) &\leq (eM)^{(k^2+k)/2}\\
 \label{eq:sumbnd}
\mbox{and} \qquad\qquad	\sum_{l=1}^k l \prod_{j=1}^l \left( 2 \binom{l}{j} M^j + 1 \right) &\le (eM)^{k^2}. 
\end{align}
If $k=1$, the upper bound of $3M$ holds in \eqref{eq:prodbnd} and \eqref{eq:sumbnd}.
\end{lemma}




\begin{proof}[Proof of Lemma \ref{prop:bounds}]
To prove the upper bound in \eqref{eq:prodbnd}, we note that 
\begin{align*}
	\prod_{j=1}^k \left( 2 \binom{k}{j} M^j + 1 \right) \leq \prod_{j=1}^k 3 \binom{k}{j} M^j = 3^k M^{k(k+1)/2} \prod_{j=1}^k \binom{k}{j}.
\end{align*}
Using falling factorial notation $(k)_j:=k(k-1)\cdots(k-j+1)$, we compute
\begin{align*}
3^k\prod_{j=1}^k \binom{k}{j} 
&= 3^k \frac{\prod_{j=1}^k j^j}{\prod_{j=1}^k j!} 
\le 3^k \frac{\prod_{j=1}^k j^j}{\prod_{j=1}^k\sqrt{2\pi j} (j/e)^j} 
= \left(\frac{3}{\sqrt{2\pi}}\right)^k \frac{\exp(k(k+1)/2)}{\sqrt{k!}}
\\ 
&\le\left(\frac{3}{\sqrt{2\pi}}\right)^k \frac{\exp(k(k+1)/2+k/2)}{(2\pi k)^{1/4}k^{k/2}}
\\ 
&= \exp\left[k(k+1)/2 + k/2 + k \log\left(3/\sqrt{2\pi}\right)- \frac k 2 \log k - \frac14\log(2\pi k)\right]\\
&\le \exp\left[k(k+1)/2 + k\Big( 1/2 +  \log\left(3/\sqrt{2\pi}\right)-  \frac 12 \log k\Big)\right].
\end{align*}
Note that the first and second inequalities above come from Stirling's approximation:
\begin{equation} \label{eq:stirling}
	j! \ge \sqrt{2\pi j} \left(\frac je\right)^j.
\end{equation}
It is easy to see that $\Big( 1/2 +  \log(3/\sqrt{2\pi})-  \frac12\log k\Big)=\log\left(\frac{3\sqrt e}{\sqrt{k2\pi}}\right)$ becomes negative for $k \ge 4$, proving the upper bound for $k\ge 4$.  For the $k=1,2,3$ cases, one can explicitly expand $\prod_{j=1}^k \left( 2 \binom{k}{j} M^j + 1 \right)$ and use $M\ge 1$ to verify that it is at most $3M$ when $k=1$ and is less than $(eM)^{(k^2+k)/2}$ when $k=2$ or $3$. 
This completes the proof of the upper bound in \eqref{eq:prodbnd}.

To show the upper bound in \eqref{eq:sumbnd}, we note that $\sum_{\ell=1}^k \ell = \frac{k^2+k}{2}$, and, for $k \ge 3$, we have $\pfrac{k^2+k}{2}(eM)^{(k^2+k)/2} \le (eM)^{k^2}$ by elementary calculus (note the function $e^{(k^2-k)/2} -  \frac{k^2+k}{2}$ is positive and increasing for all $k \ge 3$). Finally, for $k=2$, one can check that $3M+2(eM)^3 \le (eM)^4$ for any $M \ge 1$, thus proving \eqref{eq:sumbnd}.

To show the lower bound in \eqref{eq:prodbnd}, we use the same approach as for the upper bound, noting that
\begin{align*}
	\prod_{j=1}^k \left( 2 \binom{k}{j} M^j + 1 \right) \geq \prod_{j=1}^k 2 \binom{k}{j} M^j = 2^k M^{k(k+1)/2} \prod_{j=1}^k \binom{k}{j}.
\end{align*}
To bound $2^k M^{k(k+1)/2} \prod_{j=1}^k \binom{k}{j}$, we use Stirling's approximation $\displaystyle j! \le e \sqrt j \left(\frac je\right)^j$ in each place where we used \eqref{eq:stirling} in the upper bound proof above, eventually arriving at
\begin{align*}
2^k 
\prod_{j=1}^k \binom{k}{j} 
&=2^k \prod_{j=1}^k \frac{(k)_j}{j!} 
= 2^k \frac{\prod_{j=1}^k j^j}{\prod_{j=1}^k j!} \\
&\ge \exp\left[ \frac{k^2}{2} + k - \frac12 - \frac14 \log k -k\log\left(\frac{e}{2}\right) - \frac k 2 \log k \right]\\
&\ge \exp\left[ \frac{k^2}{2} - \frac k2 \log k\right],
\end{align*}
where the last inequality holds since $k - \frac12 - \frac14 \log k -k\log\left(\frac{e}{2}\right)$ is positive for all $k \ge 1$.  This completes the proof of the lower bound.

One can see from the proof that $ \frac{k^2}{2} - \frac k2 \log k + o(k \log k)$ is the correct order for the exponent on $e$ in \eqref{eq:prodbnd}.
%
%
\end{proof}

\begin{proof}[Proof of Lemma~\ref{lem:collected-bounds}]
For (i), note that $O\pfrac{1}{\sqrt n} (2e)^{\frac{\log n}{4}}$ is equal to\\ $O(\exp(-\frac12\log n + \frac{1+\log 2}{4}\log n)) = o(1)$. 

For (ii), note that $(en)^{n^{1-2\epsilon}} = \exp(n(n^{-2\epsilon}(1+\log n))= (1+o(1))^n$.

For (iii), note that 
$2e^{-n^c} (eC\sqrt n)^{n^{2c'}} \le \exp(\log 2+-n^c +n^{2c'}(1+\log (Cn)))$
$= \exp\left(n^c (-1+ \frac{\log 2}{n^c} +n^{2c'-c}(\log n + \log C)\right) \le \exp(n^c(-\frac23))$ for sufficiently large $n$.

Finally, for (iv), given $B$, note that we can choose $B'=B+2mK^2$ and then
$n^{-B'}(en^m)^{K^2} = n^{-B}\exp(-2mK^2 \log n+ K^2(1+m\log n))\le n^{-B}$ for $n$ sufficiently large.
\end{proof}

\section{Proof of Theorem \ref{thm:iidpoly}} \label{sec:iidpoly}

We will show that it is likely that $f_n(z)$ has no irreducible factor of degree $n^{1/3}/\log^3(n)$ or less, following a similar approach to Konyagin~\cite{K01}.  As in \cite{K01}, we bound the probability as a sum of two cases, depending on whether the irreducible factor is cyclotomic or not.  We have optimized the proof in the non-cyclotomic case for the highest possible degree (up to log factors); however, a stronger result can be proved in the cyclotomic case.  
Pointwise delocalization in the non-cyclotomic case will follow from an anti-concentration result proven by Tao and Vu (Lemma~\ref{lemma:loc}) which we discuss before giving the proof of Theorem~\ref{thm:iidpoly}.

Let $Z$ be a complex-valued random variable.  The \emph{L\'{e}vy concentration function} of $Z$ is defined as
$$ \Le(Z,t) := \sup_{u \in \mathbb{C}} \Prob( |Z - u| \leq t ) $$
for all $t \geq 0$.  
The L\'{e}vy concentration function bounds the \emph{small ball probabilities} for $Z$, which are the probabilities that $Z$ is in a ball of radius $t$.  

\begin{lemma}[Tao-Vu, following from Lemma 9.2 in \cite{TVpoly}] \label{lemma:loc}
Let $\xi_1, \ldots, \xi_n$ be iid Rademacher random variables, which take the values $+1$ and $-1$ with equal probability.  Let $x_0, \ldots, x_n$ be complex numbers, and suppose there is a subsequence $x_{i_1}, \ldots, x_{i_m}$ with the property that
$$ |x_{i_j}| \geq 2 |x_{i_{j+1}}| $$
for all $j=1, \ldots, m-1$.  Consider the sum $S:= \sum_{k=0}^n \xi_k x_k$.  Then one has
$$ \Le(S, 0) \leq C \exp(- cm) $$
for some absolute constants $C, c > 0$.  
\end{lemma}

We now have the tools to prove Theorem~\ref{thm:iidpoly}.

\begin{proof}[Proof of Theorem~\ref{thm:iidpoly}]
Let $\xi_0, \xi_1, \ldots$ be iid Rademacher random variables which take the values $+1$ and $-1$ with equal probability, and recall that $f_n(z) := z^n + \xi_{n-1} z^{n-1} + \cdots + \xi_1 z + \xi_0. $ The general approach below follows Konyagin~\cite{K01}.

First, we bound the number of irreducible polynomials $g(z)$ of degree $d$ that can divide $f_n(z)$.  If $g(z)$ divides $f_n(z)$, then all roots of $g(z)$ are roots of $f_n(z)$ and so are algebraic integers with absolute value between $1/2$ and 2 (by Lemma~\ref{lemma:iidpoly122} below).  Also, the set of roots of any given monic irreducible $g(z)$ are disjoint from the set of roots of any other monic irreducible polynomial (by uniqueness of the minimal polynomial and Lemma~\ref{lemma:minpoly}); thus, the number of degree $d$ algebraic integers that can be roots of $f_n(z)$ is an upper bound for the number of possible degree $d$ irreducible polynomials $g(z)$ that can divide $f_n(z)$. Hence, by Lemmas~\ref{lemma:counting} and~\ref{prop:bounds}, we have 
\begin{equation}\label{eqn:numdivisors}
\#\{\mbox{degree $d$ irreducible } g(z) \mbox{ that divide } f_n(z)\} \le  (2e)^{d^2}.
\end{equation}
When $d \le 2$, the total number of possible divisors $g(z)$ is a constant, so applying Remark~\ref{rem:weaker_iidpoly}, the probability that any such polynomial divides $f_n(z)$ is at most $O(1/\sqrt n)$.  Thus, it is sufficient to consider irreducible divisors with degree at least 3, and we assume for the remainder of the proof that a possible divisor $g(z)$ has degree $d \ge 3$.

For the non-cyclotomic case, we will show  sufficient delocalization so that the probability that $f_n(z)$ has an irreducible factor of degree $d$ with $3 \le d \le n^{1/3}/\log^3(n)$ is exponentially small.  
Let $g(z)$ be an abritrary non-cyclotomic irreducible polynomial with degree $d$ where $3 \le d \le n^{1/3}/\log^3(n)$.  
By Lemmas~\ref{lemma:minpoly} and \ref{lemma:alginteger}, we may assume $g$ is monic; also, $g(z)$ divides $f(z)$ if and only if $f(w) = 0$, where $w$ is any root of $g(z)$.  By a result of Dobrowolski~\cite{Dob}, since $g(z)$ is non-cyclotomic, it must have a root $w$ satisfying
$$\abs{w} \ge 1 + \frac{c}{d}\pfrac{\log\log d}{\log d}^3,$$
where $c$ is a positive constant; note that the lower bound strictly exceeds 1 since $d \ge 3$.  We will show that the sequence $1, w, w^2, w^3, \dots, w^n$ contains a subsequence that grows quickly, and then apply Lemma~\ref{lemma:loc}.  Because $d \le n^{1/3}/\log^3 n$, we know that $\abs w \ge 1 + \frac{c}{n^{1/3}}$ and so if we take a minimal integer $b$ satisfying $b \ge \frac{4n^{1/3}}{c}$, we have that  $\abs{w}^{b} \ge (1 + \frac{c}{n^{1/3}})^b \ge \exp\left(\frac{cb}{2n^{1/3}}\right)$
for sufficiently large $n$ 
(since $1+x \ge e^{x/2}$ for all sufficiently small positive $x$)%
, which shows that $\abs{w}^b \ge 2$.  We can now take the subsequence $w^0, w^b, w^{2b}, \dots, w^{\floor{\frac{n}{b}}b}$, noting each term is at least twice the term before in absolute value, and so Lemma~\ref{lemma:loc} implies the delocalization bound  
\begin{align}
\Pr(g(z) \mbox{ divides } f_n(z)) &= \Pr(f_n(w) = 0)\nonumber 
\\ &
 \le C\exp\left(-c\left(\floor{\frac{n}{b}}+1\right)\right) \le C\exp(-cn^{2/3}),\nonumber \label{eqn:pm1deloc}
\end{align} 
where $C$ and $c$ are constants that may change from line to line.  
%
By \eqref{eqn:numdivisors}, there are at most $(2e)^{d^2} \leq \exp\left(\frac{n^{2/3}}{\log^6 n}(1+\log 2)\right)$ possible polynomials $g(z)$ with degree $d$ where $3\le d \le n^{1/3}/\log^3 n$.
Taking a union bound over all possible $g(z)$ with all possible degrees $d$, we see that the probability that $f_n(z)$ has an irreducible non-cyclotomic factor with degree in the range $3\le d \le n^{1/3}/\log^3 n$ is at most $C\frac{n^{1/3}}{\log^3 n} \exp\left(-cn^{2/3}+\frac{n^{2/3}}{\log^6 n}(1+\log 2) \right)$, which is less than $\exp(-c n^{2/3})$ for sufficiently large $n$.

For the cyclotomic case, we will use the union bound over all cyclotomic polynomials of a given degree. Let $E_{\mathrm{cyl}}$ be the event that there exists a cyclotomic polynomial of degree $d$ with $3 \le d \le n^{1/3}$ that divides $f_n(x)$, and let $E_{\mathrm{cyl},d}$ be the event that there exists a cyclotomic polynomial of degree $d$ that divides $f_n(x)$.  By the union bound, we have
$\Pr(E_{\mathrm{cyl}}) \le \sum_{d =3}^{n^{1/3}} \Pr(E_{\mathrm{cyl},d}),$
and so it remains to prove a bound on $\Pr(E_{\mathrm{cyl},d})$.


Recall that the $k$-th cyclotomic polynomial is $\Phi_k(x) = \prod_{\substack{1\le a\le k \\ \gcd(a,k)=1}} \left(x - e^{2\pi i \frac{a}{k}}\right)$.  Assume that $\Phi_k(x)$ has degree $d\ge 3$ (which implies $k \ge 3$) and divides $f_n(x)$.  Then we have $f_n(\alpha)=0$, where $\alpha$ is a root of $\Phi_k(x)$.
Because $\alpha^k=1$, we have that $f_n(\alpha) = \sum_{j=0}^{k-1} A_j \alpha^j$ where $A_j = \sum_{\substack{b \equiv j \pmod k \\ 0\le b \le n}} \xi_b$; note that the $A_j$ are independent.  Because $f_n(\alpha) =0$ and $\alpha$ has algebraic degree $d$, we have $\sum_{j=0}^{d-1} A_j \alpha^j = -\sum_{j=d}^{k-1} A_j \alpha^j = \sum_{j=0}^{d-1} B_j \alpha^j$, for some integers $B_j$ that are functions of $A_{d},\dots, A_{k-1}$, and so the $B_j$ are independent of $A_0,\dots, A_{d-1}$.  Furthermore, because the minimal polynomial for $\alpha$ has degree $d$, we must have $A_j = B_j$ for $0\le j \le d -1$. We may condition on the $B_j$ and apply the Littlewood-Offord inequality (see, for example, \cite[Corollary~7.8]{TVaddcomb}) to each equation $A_j = \sum_{\substack{b \equiv j \pmod k \\ 0\le b \le n}} \xi_b = B_j$, resulting in the bound 
\[ \Pr(\Phi_k(x) \mbox{ divides } f_n(x)) \le \Pr(A_j=B_j\mbox{ for } 0\le j \le d -1) = O\pfrac{\sqrt{k}}{\sqrt n}^d. \]
It is well-known that if $\Phi_k(x)$ has degree $d$, then $k \le c d \log \log d$ for a constant $c$ (following from, for example, Rosser and Schoenfeld \cite[Theorem~15]{RS}), and so using the assumption that $d \le n^{1/3}$, we have 
\[ \Pr(\Phi_k(x) \mbox{ divides } f_n(x)) \le O\pfrac{1}{n^{1/4}}^d, \]
a bound independent of $k$.  Because every cyclotomic polynomial is equal to $\Phi_k(x)$ for some $k$, we have that $\Pr(E_{\mathrm{cyl},d}) \le O\pfrac{1}{n^{1/4}}^d N(d)$, where $N(d)$ is the number of cyclotomic polynomials with degree $d$.  Pomerance \cite{Pom} showed that $N(d) \le c d$ (in fact, \cite{Pom} shows that the constant $c$ tends slowly to zero as $d$ tends to infinity).  Thus, $\Prob(E_{\mathrm{cyl},d}) \le d O\pfrac{1}{n^{1/4}}^d$.

We now apply the union bound to $E_{\mathrm{cyl}}= \bigcup_{3\le d \le n^{1/3}} E_{\mathrm{cyl},d}$.  From the discussion after \eqref{eqn:numdivisors}, we know that the probability of a factor with degree at most 2 is bounded by $O(1/\sqrt n)$ (in fact, this bound is tight for the possible factors $x+1$ and $x-1$), and so by the previous paragraph, we need a similar bound on $\sum_{d=3}^{n^{1/3}} \pfrac{c}{n^{1/4}}^d d$ for any constant $c$.   In fact, using the formula for an infinite arithmetico--geometric series, we have that 
\[ \sum_{d=3}^{\infty} \pfrac{c}{n^{1/4}}^d d \le 4\pfrac{c}{n^{1/4}}^3\le O\pfrac{1}{\sqrt n} \]
for $n$ sufficiently large, which completes the proof.
\end{proof}

\begin{lemma} \label{lemma:iidpoly122}
If $ f_n(z) := z^n + \xi_{n-1} z^{n-1} + \cdots + \xi_1 z + \xi_0 $ is a polynomial in which the coefficients 
$\xi_0, \xi_1, \ldots \xi_{n-1}$ take values $ 1$  or $-1$, all roots of $f_n$ have absolute value strictly between $1/2$ and $2$.
\end{lemma}
\begin{proof}
If $|z| \le 1/2$, then
$	\left|z^n+ \sum_{j=1}^{n-1} z^j \xi_j \right| \leq \sum_{j=1}^n \frac{1}{2^j} < 1,$
and hence
$ |f_n(z)| \geq |\xi_0| - \left|z^n+ \sum_{j=1}^{n-1} z^j \xi_j \right| > 0. $
Thus, if $f_n(z)=0$, we must have $\abs z > 1/2$.  Similarly, if $|z| \ge 2$, then $\left| \sum_{j=0}^{n-1} z^j \xi_j \right| \leq \sum_{j=0}^{n-1} 2^j < 2^n,$ and hence 
$ |f_n(z)| \geq 2^n - \left| \sum_{j=1}^n z^j \xi_j \right| > 0,$ showing that $\abs z < 2$ for any value of $z$ for which $f_n(z) = 0$.
%
\end{proof}

\section*{Acknowledgments}

We thank Melanie Matchett Wood for many useful conversations and for contributing key ideas for Theorem~\ref{thm:main}.  The first author thanks Peter D.T.A. Elliott, Richard Green, and Katherine Stange for useful discussions and references.  The second author thanks Van Vu for originally suggesting this line of inquiry.  The authors also thank the anonymous referee for useful comments and suggestions which led to the current version of Theorem~\ref{thm:iidpoly}.  Christian Borst, 
Evan Boyd, 
Claire Brekken, 
and 
Samantha Solberg 
produced Figure~\ref{fig:pm1}
and were supported by NSF grant DMS-1301690 and co-supervised by Melanie Matchett Wood.  The second author thanks Steve Goldstein for helping direct Borst, Boyd, Brekken, and Solberg's research.   The second author also thanks the Simons Foundation for providing Magma licenses and the Center for High Throughput Computing (CHTC) at the University of Wisconsin-Madison for providing computer resources.

\end{document}